\documentclass[11pt, oneside]{article}

\usepackage{geometry}                		
\usepackage{amssymb}
\usepackage{amsthm}
\usepackage{amsmath}

\usepackage{verbatim}

\def\F {{\mathcal F}}

\def\S {{\mathcal S}}
\def\W {{\mathcal W}}

\newtheorem{thm}{Theorem}[section]
\newtheorem{lemma}[thm]{Lemma}

\newtheorem{cor}[thm]{Corollary}
\newtheorem{conj}[thm]{Conjecture}

\newtheorem{claim}[thm]{Claim}

\makeatletter
\@addtoreset{case}{thm}
\makeatletter

\theoremstyle{definition}

\title{Minimum Size of Feedback Vertex Sets of Planar Graphs of Girth at least Five}

\author{Tom Kelly\thanks{Department of Combinatorics and Optimization, University of Waterloo, Waterloo, Ontario, Canada. Email: \texttt{t9kelly@uwaterloo.ca}}
\and
Chun-Hung Liu\thanks{Department of Mathematics, Princeton University, Princeton, New Jersey, USA. Email: \texttt{chliu@math.princeton.edu}}
}

\date{\today}

\begin{document}
\maketitle

\begin{abstract}
A feedback vertex set of a graph is a subset of vertices intersecting all cycles.
We provide tight upper bounds on the size of a minimum feedback vertex set in planar graphs of girth at least five.  
We prove that if $G$ is a connected planar graph of girth at least five on $n$ vertices and $m$ edges, then $G$ has a feedback vertex set of size at most $\frac{2m-n+2}{7}$.  
By Euler's formula, this implies that $G$ has a feedback vertex set of size at most $\frac{m}{5}$ and $\frac{n-2}{3}$.  
These results not only improve a result of Dross, Montassier and Pinlou and confirm the girth-5 case of one of their conjectures, but also make the best known progress towards a conjecture of Kowalik, Lu\v{z}ar and \v{S}krekovski and solves the subcubic case of their conjecture.
An important step of our proof is providing an upper bound on the size of minimum feedback vertex sets of subcubic graphs with girth at least five with no induced subdivision of members of a finite family of non-planar graphs.
\end{abstract}

\section{Introduction}

In this paper, graphs are simple.
A {\em feedback vertex set} of a graph $G$ is a set $S\subseteq V(G)$ such that $G-S$ is a forest.  
We define $\phi(G)$ to be the minimum size of a feedback vertex set of a graph $G$.  
Feedback vertex sets have been extensively studied.
For example, given a graph $G$ and an integer $k$, deciding if $\phi(G) \leq k$ is one of Karp's original NP-complete problems \cite{K72}.  

A direction in the study of feedback vertex sets is to find an upper bound on the minimum size of a feedback vertex set (for example, see \cite{amt, p, s}).
An old conjecture due to Albertson and Berman \cite{AB79} states that every planar graph on $n$ vertices has a feedback vertex set of size at most $\frac{n}{2}$.
This conjecture, if true, implies that every planar graph on $n$ vertices has an independent set of size at least $\frac{n}{4}$.
This bound for the independent set is an immediate corollary of the Four Color Theorem, but it is the only proof in the literature.
Albertson and Berman's conjecture remains open.
The best known result is that every planar graph $G$ on $n$ vertices has a feedback vertex set with size at most $\frac{3n}{5}$, due to Borodin's acyclic 5-coloring theorem \cite{B76}.

The {\it girth} of a graph is the length of its shortest cycle.
(If the graph has no cycle, then its girth is infinity.)
Recently, the minimum size of feedback vertex sets of planar graphs with girth at least five has attracted attention.
Kowalik et al.\ \cite{KLS10} proposed the following conjecture, which is tight as the dodecahedron attains the bound.

\begin{conj}[\cite{KLS10}] \label{girth 5 vertices conj}
If $G$ is a planar graph of girth at least five on $n$ vertices, then $\phi(G)\leq\frac{3n}{10}$.
\end{conj}

Dross et al.\ \cite{DMP14} proved that if $G$ is a planar graph of girth at least five with $m$ edges, then $\phi(G) \leq \frac{5m}{23}$.
Note that every planar graph on $n$ vertices with girth at least five has at most $\frac{5}{3}(n-2)$ edges by Euler's formula (when $n \geq 4$).
So every planar graph with girth at least five on $n$ vertices has a feedback vertex set of size at most $\frac{25n-50}{69}$ (when $n \geq 4$).

In a companion paper, Dross et al.\ \cite{DMP15} made the following conjecture.

\begin{conj}[\cite{DMP15}] \label{high girth conj}
If $G$ is a planar graph of girth at least $g$ with $m$ edges, then $\phi(G) \leq \frac{m}{g}$.
\end{conj}

In this paper, we prove the following.

\begin{thm}\label{main thm}
If $G$ is a connected planar graph of girth at least five on $n$ vertices and $m$ edges, then $\phi(G)\leq \frac{2m - n + 2}{7}$.
\end{thm}

A graph is {\it subcubic} if every vertex has degree at most three.
It is easy to see that every subcubic graph on $n$ vertices has at most $\frac{3n}{2}$ edges.
Together with the fact that every non-tree planar graph on $n$ vertices with girth at least five has at most $\frac{5}{3}(n-2)$ edges, the following are immediate corollaries of Theorem \ref{main thm}.

\begin{cor}\label{planar cor edges}
Let $G$ be a graph on $n$ vertices and $m$ edges with $n \geq 2$.
	\begin{enumerate}
		\item If $G$ is a planar graph with girth at least five, then $\phi(G) \leq \frac{m}{5}$, and $\phi(G) \leq \frac{n-2}{3}$.
		\item If $G$ is a planar subcubic graph with girth at least five, then $\phi(G) \leq \frac{3n}{10}$.
		\item If $G$ is a connected planar subcubic graph with girth at least five, then $\phi(G) \leq \frac{2n+2}{7}$.
	\end{enumerate}
\end{cor}

The first statement of Corollary \ref{planar cor edges} improves the result of Dross et al.\ about Conjecture \ref{girth 5 vertices conj} mentioned earlier and solves Conjecture \ref{high girth conj} for the case $g=5$.
The second or the third statement of Corollary \ref{planar cor edges} solves Conjecture \ref{girth 5 vertices conj} for the case when $G$ is subcubic.
Note that the dodecahedron attains the bound in Theorem \ref{main thm}.

Our strategy for proving Theorem \ref{main thm} is first proving the case when $G$ is a subcubic graph, and then boosting it to the general case.
In fact, instead of proving the subcubic case of Theorem \ref{main thm}, we will prove a stronger result: the planarity will be replaced by the property of the lack of induced subdivision of graphs in a finite family of non-planar graphs.
We say that a graph $G$ contains an {\it induced subdivision} of another graph $H$ if $G$ contains an induced subgraph that can be obtained from $H$ by repeatedly subdividing edges.
For every even number $n$ with $n \geq 6$, let $M_n$ be the cubic graph obtained from the $n$-cycle by adding edges connecting each opposite pair of vertices.
Notice that $M_n$ is not planar but any of its proper induced subgraphs is.
So unlike the minor relation or the subdivision relation, there is no Kuratowski-type theorem for planarity with respect to the induced subdivision relation, even if we restrict the problem to subcubic graphs.

Furthermore, in order to make the induction go through, we relax the girth condition to allow the existence of short cycles, but we do not allow two disjoint short cycles. 
The following is our result about subcubic graphs.
The family $\F$ and the function $r$ mentioned in the description will be explicitly described in this paper.

\begin{thm} \label{strong subcubic}
There exists a finite family $\F$ of non-planar 2-connected subcubic graphs and a function $r$ such that if $G$ is a connected subcubic graph with no two disjoint cycles of length less than five and $G$ does not contain an induced subdivision of a member of $\F$, then $\phi(G)\leq\frac{2\lvert E(G) \rvert - \lvert V(G) \rvert + 2}{7} + r(G)$, and $r(G)$ satisfies the following.
	\begin{enumerate}
		\item $r(G) \leq \frac{4}{7}$. 
		\item If $G \neq K_4$, then $r(G) \leq \frac{3}{7}$; if $r(G) = \frac{3}{7}$, then $G$ contains a triangle.
		\item If $r(G) > 0$ and $G$ is 2-connected, then for all $e\in E(G)$, $\phi(G-e)\leq\frac{2\lvert E(G) \rvert - \lvert V(G) \rvert - 5}{7} + r(G)$.
		\item If $G$ is planar and has girth at least five, then $r(G)=0$.
		\end{enumerate}
\end{thm}
The proof of Theorem \ref{strong subcubic} requires explicit descriptions of the family $\F$ and the function $r$.
However, proving the general case is relatively simple as long as the subcubic case is established.
In Section \ref{sec:reduction}, we show how to prove Theorem \ref{main thm} assuming that Theorem \ref{strong subcubic} is proved.  
In Section \ref{sec:r}, we explicitly define the function $r$.  
In Section \ref{sec:subcubic proof}, we complete the paper by proving Theorem \ref{strong subcubic}.
The proof of Theorem \ref{strong subcubic} will extensively use strategies in one of our earlier papers \cite{kl}.

Here are some notions that will be frequently used in this paper.
Let $G$ be a graph and $X$ a subset of $V(G)$.
We define $G[X]$ to be the subgraph induced on $X$, and we define $G-X$ to be the subgraph induced on $V(G)-X$.
When $X$ consists of one vertex, say $v$, then we also denote $G-X$ by $G-v$.
A vertex $v$ in a graph $G$ is a {\em cut-vertex} if $G- v$ has more connected components than $G$.  
Similarly, an edge $e$ in $G$ is a {\em cut-edge} if $G- e$ has more connected components than $G$, where $G-e$ is the subgraph of $G$ obtained from $G$ by deleting the edge $e$.
For a positive integer $k$, a graph is {\it $k$-connected} if it has at least $k+1$ vertices and $G-Y$ is connected for every subset $Y$ of $V(G)$ with size at most $k-1$; a graph is {\it $k$-edge-connected} if $G-Y$ is connected for every subset $Y$ of $E(G)$ of size at most $k-1$.
A {\em block} of $G$ is a maximal connected subgraph without a cut-vertex.  
A block is {\em nontrivial} if it is not isomorphic to $K_1$ or $K_2$.  
Note that every non-trivial block is 2-connected and has minimum degree at least two.
So in a subcubic graph, every pair of nontrivial blocks is disjoint.
An {\em end-block} of $G$ is a block containing at most one cut-vertex of $G$.
Note also that every non-2-connected graph on at least three vertices contains at least two end-blocks.
For a pair of non-adjacent vertices $x,y$ of $G$, we define $G+xy$ to be the graph obtained from $G$ by adding the edge $xy$.
For a family ${\mathcal W}$ of multigraphs, we say that a multigraph $G$ is ${\mathcal W}$-free if it does not contain any induced subgraph isomorphic to a member of ${\mathcal W}$.


\section{Reducing Theorem \ref{main thm} to subcubic graphs}\label{sec:reduction}

In this section, we prove Theorem \ref{main thm}, assuming Theorem \ref{strong subcubic}.

\bigskip

\noindent{\bf Proof of Theorem \ref{main thm} assuming Theorem \ref{strong subcubic}.}
In the rest of this section, $G$ denotes a minimum counterexample to Theorem \ref{main thm}.  
That is, $G$ is connected, planar and of girth at least five with $\phi(G) > \frac{2\lvert E(G) \rvert-\lvert V(G) \rvert+2}{7}$, but for every connected planar graph $H$ with girth at least five with $\lvert V(H) \rvert + \lvert E(H) \rvert< \lvert V(G) \rvert + \lvert E(G) \rvert$, $\phi(H) \leq \frac{2\lvert E(H) \rvert - \lvert V(H) \rvert +2}{7}$.
We denote $\lvert V(G) \rvert$ and $\lvert E(G) \rvert$ by $n$ and $m$, respectively.

Since every graph in the family $\F$ mentioned in Theorem \ref{strong subcubic} is non-planar, $G$ does not contain an induced subdivision of a member of $\F$.
Hence $G$ contains a vertex of degree at least four by Theorem \ref{strong subcubic}.
In particular, $G$ contains at least five vertices.

\begin{claim}
$G$ is 2-edge-connected.
\end{claim}

\begin{proof}
Suppose to the contrary that $e\in E(G)$ is a cut-edge and $G-e$ has components $G_1$ and $G_2$.  By the minimality of $G$, for $i=1,2$, $G_i$ admits a feedback vertex set $S_i$ of size at most 
$\frac{2|E(G_i)| - |V(G_i)| + 2}{7}$.
Then $S_1\cup S_2$ is a feedback vertex set of $G$ of size at most 
$\frac{2(|E(G_1)| + |E(G_2)|) - (|V(G_1)| + |V(G_2)|) + 4}{7} \leq \frac{2m - n + 2}{7}$,
a contradiction.
\end{proof}

\begin{claim}\label{max degree four}
$G$ has maximum degree four, and every vertex of degree four is a cut-vertex.
\end{claim}

\begin{proof}
Let $v$ be a vertex of $G$ with degree at least four.  
Denote the degree of $v$ by $d(v)$.
Let $G_1,\dots, G_k$ be the components of $G-v$.  
Since $G$ is 2-edge-connected, $k\leq \lfloor \frac{d(v)}{2}\rfloor$.  
By the minimality of $G$, for $i=1,\dots, k$, $G_i$ admits a feedback vertex set $S_i$ of size at most 
$\frac{2|E(G_i)| - |V(G_i)| + 2}{7}$.
Then $S_1\cup \cdots \cup S_k\cup\{v\}$ is a feedback vertex set of $G$ of size at most 
$$1 + \sum_{i=1}^k \frac{2|E(G_i)| - |V(G_i)| + 2}{7} \leq \frac{2(m - d(v)) - (n - 1) + 2k}{7} + 1.$$
If $d(v) \geq 5$, then $d(v)-k \geq d(v)-\lfloor \frac{d(v)}{2}\rfloor \geq 3$, so $\phi(G) \leq \frac{2m-n+2}{7}$, a contradiction.
If $d(v)=4$ but $G-v$ is connected, then $k=1$ and $\phi(G)\leq \frac{2m-n+2}{7}$, a contradiction.
\end{proof}

By Claim \ref{max degree four}, $G$ is not 2-connected and contains at least five vertices.
So $G$ has at least two end-blocks.  
Let $B$ be an end-block of $G$.  
Let $n_B  = |V(B)|$ and $m_B = |E(B)|$.  
Since $B$ is an end-block of $G$, $B$ contains a unique $v\in V(G)$ such that $v$ is a cut-vertex in $G$.  
Since $G$ is 2-edge-connected, $v$ has degree four in $G$ and has precisely two neighbors, denoted by $v_1$ and $v_2$, in $B$.  
Note that $B$ is 2-connected, so $B$ is subcubic by Claim \ref{max degree four}.
Furthermore, $v_1$ is not adjacent to $v_2$, otherwise $vv_1v_2v$ is a triangle in $G$, a contradiction.

Let $B' = G[V(B)-\{v\}]+v_1v_2$.  
Since $B$ has girth at least five, every cycle in $B'$ of length less than five contains $v_1v_2$, so $B'$ does not contain disjoint cycles of length less than five.  Since $B'$ is subcubic, by Theorem \ref{strong subcubic}, $B'$ admits a feedback vertex set $S_1$ of size at most
$$\frac{2(m_B - 1) - (n_B - 1) + 2}{7} + r(B') = \frac{2m_B - n_B + 1}{7} + r(B').$$

Let $G' = G-(V(B)-\{v\})$.  By the minimality of $G$, $G'$ admits a feedback vertex set $S_2$ of size at most
$$\frac{2(m - m_B) - (n - n_B + 1) + 2}{7}  = \frac{2(m - m_B) - (n - n_B) + 1}{7}.$$
Therefore $S_1\cup S_2$ is a feedback vertex set of $G$ of size at most $\frac{2m - n + 2}{7} + r(B')$.
Since $G$ is a counterexample, $r(B') > 0$.  
Since $B'$ is 2-connected, by Theorem \ref{strong subcubic}, $B'-v_1v_2$ admits a feedback vertex set $S'_1$ of size at most $\frac{2m_B - n_B -5}{7} + r(B')$.
Since there are only two edges between $v$ and $G-V(B)$, and $G$ is 2-edge-connected, $G-V(B)$ is connected.
By the minimality of $G$, $G-V(B)$ admits a feedback vertex set $S'_2$ of size at most
$$\frac{2(m - m_B - 2) - (n - n_B) + 2}{7} = \frac{2(m - m_B) - (n - n_B) - 2}{7}.$$
Therefore, $S'_1\cup S'_2\cup\{v\}$ is a feedback vertex set of $G$ of size at most $\frac{2m - n}{7} + r(B')$.
So $r(B') > \frac{2}{7}$.  
By Theorem \ref{strong subcubic}, either $B'=K_4$ or $B'$ contains a triangle.  
Hence $B'$ contains a triangle.
However, $B$ can be obtained from $B'$ by subdividing an edge, so $B$ contains a 4-cycle, contradicting that $G$ has girth at least five.
This proves Theorem \ref{main thm}.


\section{Error Terms}\label{sec:r}
In this section we will describe the function $r$ mentioned in Theorem \ref{strong subcubic}.
We shall first define special families of graphs.

Let $G$ be a graph.
For $e_1,e_2 \in E(G)$ (not necessarily distinct) and $a \in V(G)$ with degree two, we define the graph $G \circ (e_1,e_2,a)$ to be the graph obtained from $G$ by subdividing $e_1$ and $e_2$ once, respectively, and adding a new vertex adjacent to $a$ and the two vertices obtained from subdividing $e_1$ and $e_2$.
Note that when $e_1=e_2$, it is subdivided twice.

Define $\F_{1,0}$ to be the family consisting of the multigraph on one vertex with one loop, and we define $\F_{i,j}=\emptyset$ if $i\leq0$ or $j<0$.
For $0\leq j \leq i$, we define $\F_{i,j}$ to be the set of subcubic multigraphs that can be either obtained from a multigraph in $\F_{i-1,j}$ by subdividing an edge once or obtained from a graph in $\F_{i,j-1}$ by taking operation $\circ$.

In particular, for each $i \geq 1$, the unique member in $\F_{i,0}$ is the cycle of length $i$; the unique graph in $\F_{1,1}$ is $K_4$; the unique graph in $\F_{2, 1}$ is the graph obtained from $K_4$ by subdividing an edge, and we denote this graph by $K_4^+$.
Furthermore, there is only one graph in $\F_{3,1}$ with girth at least four, which is the graph obtained from $K_4$ by subdividing each edge once in a perfect matching of $K_4$.

Note that every graph in $\F_{i,j}$ has minimum degree at least two for every $i,j$.
Since the operation $\circ$ decreases the number of vertices of degree two by one, for every $i,j$, each member of $\F_{i,j}$ has exactly $i-j$ vertices of degree two.
Hence every graph in $\F_{i,j}$ is cubic if and only if $i = j$.

There are two special graphs in $\F_{2,2}$ worth mentioning.
One is the cube, denoted by $Q_3$, which is the planar triangle-free cubic graph obtained from two disjoint cycles of length four by adding a perfect matching.
The other is the Wagner graph, denoted by $V_8$, which is the cubic graph obtained from the cycle of length eight by adding edges between pairs of vertices of distance four.  

The following are some properties of graphs in the families $\F_{i,j}$.
They can be easily verified and we omit the proof.

\begin{lemma} \label{basic F}
Let $i,j$ be integers with $i \geq 1$ and $0\leq j \leq i$.
The following holds.
	\begin{enumerate}
		\item $\F_{i,0}$ consists of the cycle of length $i$.
		\item All members of $\F_{i,j}$ are simple, unless $i \in \{1, 2\}$ and $j=0$.
		\item $\F_{1,1}=\{K_4\}$.
		\item $\F_{2,1}=\{K_4^+\}$.
		\item $\F_{3,1}$ consists of three graphs, where each of them is obtained from $K_4$ by either subdividing an edge twice or subdividing two edges once; 
			the graph obtained from $K_4$ by subdividing the edges of a perfect matching is the only graph in $\F_{3,1}$ with girth at least four.
		\item\label{F_{2,2}} Every graph in $\F_{2,2}$ contains disjoint cycles of length less than five; every graph in $\F_{2,2}\backslash\{Q_3,V_8\}$ contains a triangle disjoint from a cycle of length at most four.
		\item Every graph in $\F_{3,2} \cup \F_{4,1}$ has girth at most four.
		\item If $G\in\F_{i,j}$, then $|V(G)| = i+3j$, $|E(G)| = i + 5j$, and $\frac{2\lvert E(G) \rvert -\lvert V(G) \rvert+2}{7} = \frac{i+2}{7}+j$.
	\end{enumerate}
\end{lemma}
Note that Statement 8 of Lemma \ref{basic F} implies that $\F_{i,j} \cap \F_{i',j'} = \emptyset$ if $(i,j) \neq (i',j')$.

The following lemmas were proved in one of our earlier papers \cite{kl}.
They show that if a graph has the property that deleting an arbitrary set of edges with size at most one can make the size of the minimum feedback vertex set small, then so is the graph obtained from it by taking the two operations used for defining the families $\F_{i,j}$.
In particular, they provide a recurrence relation to upper bound the size of minimum feedback vertex sets of graphs in $\F_{i,j}$.

\begin{lemma}[\cite{kl}] \label{lemma:subdivideedgestrong}
Let $H$ be a graph, $k\in \{0,1\}$, and let $G$ be obtained from $H$ by subdividing an edge.  
If there exists $p\in\mathbb R$ such that $\phi(H- W)\leq p$ for every $W\subseteq E(H)$ with $|W| = k$, then $\phi(G- W')\leq p$ for every $W'\subseteq E(G)$ with $|W'| = k$.
\end{lemma}

\begin{lemma}[\cite{kl}] \label{lemma:op1edgestrong}
Let $H$ be a graph, $k\in\{0,1\}$, $a\in V(H), e_1,e_2\in E(H)$, and $G = H\circ (e_1,e_2,a)$.  
If there exists $p\in\mathbb R$ such that $\phi(H- W) \leq p$ for every $W\subseteq E(H)$ with $|W| = k$, then $\phi(G- W')\leq p + 1$ for every $W'\subseteq E(G)$ with $|W'| = k$.
\end{lemma}

Now we are ready to describe the function $r$ mentioned in Theorem \ref{strong subcubic}.
Let $H$ be a 2-connected graph or a multigraph on at most two vertices.  
We define
$$\epsilon(H) = \left\{
\begin{array}{l l}
\max\{\frac{5 - i}{7},0\} & \mbox{if } H\in \F_{i,j},i \geq 1, j \geq 0,\\
0 & \mbox{otherwise.}
\end{array}\right.$$
Let $G$ be a connected multigraph that is either simple or a member of $\F_{i,j}$ for some nonnegative integers $i,j$.
We define $r(G)$ to be the sum of $\epsilon(B)$ over all blocks $B$ of $G$ if $G$ is simple; we define $r(G)=\epsilon(G)$ if $G \in \F_{i,j}$ for some nonnegative integers $i,j$.
Note that every simple graph belonging to $\F_{i,j}$ for some $i$ and $j$ is 2-connected, so $r$ is well-defined, and $r(G)=\epsilon(G)$ if $G \in \F_{i,j}$.

\begin{lemma} \label{edge property F}
If $G \in \F_{i,j}$ for some nonnegative integers $i,j$, then $\phi(G-e) \leq \frac{2|E(G)| - |V(G)| + 2}{7} + \epsilon(G) - 1$ for every edge $e \in E(G)$.
\end{lemma}

\begin{proof}
We shall prove this lemma by induction on $i+j$.
When $i=1,j=0$, $G$ has only one edge $e$, so $\phi(G-e)=0\leq \frac{2|E(G)| - |V(G)| + 2}{7} - 1 + \epsilon(G)$, as desired.
Therefore we may assume that the lemma holds for every pair of $i',j'$ with $i'+j'<i+j$.

We first assume that $G$ is obtained from a multigraph $H$ in $\F_{i-1,j}$ by subdividing an edge.
By the induction hypothesis, $\phi(H-f) \leq \frac{2|E(H)| - |V(H)| + 2}{7} + \epsilon(H) - 1$  for every $f \in E(H)$.
By Lemma \ref{lemma:subdivideedgestrong}, for every $e \in E(G)$, 
$$\phi(G-e) \leq \frac{2|E(H)| - |V(H)| + 2}{7} + \epsilon(H) - 1 \leq \frac{2|E(G)| - |V(G)| + 2}{7} + \epsilon(G) -1,$$
as desired.

So we may assume that $G$ is obtained from a multigraph $H$ in $\F_{i,j-1}$ by taking the operation $\circ$.
By the induction hypothesis, $\phi(H-f) \leq \frac{2|E(H)| - |V(H)| + 2}{7} +\epsilon(H) - 1$ for every $f \in E(H)$.
By Lemma \ref{lemma:op1edgestrong}, for every $e \in E(G)$, 
$$\phi(G-e) \leq \frac{2|E(H)| - |V(H)| + 2}{7} + \epsilon(H) \leq \frac{2|E(G)| - |V(G)| + 2}{7} + \epsilon(G) - 1,$$
as desired.
\end{proof}

\begin{lemma} \label{vertex property F}
If $G \in \F_{i,j}$, then $\phi(G-v) \leq \frac{2|E(G)| - |V(G)| + 2}{7} + \epsilon(G) - 1$ for every vertex $v \in V(G)$.
In particular, for every vertex $v \in V(G)$, there exists a feedback vertex set $S$ of $G$ such that $v \in S$ and $\lvert S \rvert \leq \frac{2|E(G)| - |V(G)| + 2}{7} + \epsilon(G)$.
\end{lemma}

\begin{proof}
Let $v$ be a vertex of $G$ and let $e$ be an edge of $G$ incident with $v$.
Then $\phi(G-v) \leq \phi(G-e) \leq \frac{2|E(G)| - |V(G)| + 2}{7} + \epsilon(G) - 1$ by Lemma \ref{edge property F}.
\end{proof}

\begin{lemma}\label{girth 5 planar in F}
If $\F_{i,j}$ contains a planar graph with girth at least five, then $i \geq 5$.
\end{lemma}

\begin{proof}
Let $G$ be a planar graph in $\F_{i,j}$ with girth at least five.  
Since $G$ is not a tree, by Euler's formula, $|E(G)| \leq \frac{5}{3}(|V(G)| - 2)$.  
By Lemma \ref{basic F}, $i + 5j \leq \frac{5}{3}(i + 3j - 2)$.
That is, $i\geq 5$.
\end{proof}

Now we are ready to prove that the function $r$ satisfies the conditions in Theorem \ref{strong subcubic}.

\begin{lemma} \label{r is good}
Let $\W$ be the set of graphs in $\F_{4,2} \cup \F_{4,3}$ with girth at least five.
If $G$ is a connected $\W$-free subcubic graph with no two disjoint cycles with length less than five, then the following hold.
	\begin{enumerate}
		\item $r(G) \leq \frac{4}{7}$. 
		\item If $G \neq K_4$, then $r(G) \leq \frac{3}{7}$; if $r(G) = \frac{3}{7}$, then $G$ contains a triangle.
		\item If $r(G) > 0$ and $G$ is 2-connected, then for all $e\in E(G)$, $\phi(G-e)\leq\frac{2m - n - 5}{7} + r(G)$.
		\item If $G\not \in \F_{3,3} \cup \F_{4,4}$ and has girth at least five, then $r(G)=0$.
	\end{enumerate}
\end{lemma}

\begin{proof}
We first assume that $G$ is 2-connected or has at most two vertices.
If $r(G) > 0$, then $G \in \F_{i,j}$ for some $i,j$ with $1 \leq i \leq 4$.
So $r(G) \leq \frac{5-i}{7} \leq \frac{4}{7}$.
This proves Statement 1.
If $r(G) \geq \frac{3}{7}$ and $G \neq K_4$, then $i=2$, so by Lemma \ref{basic F}, either $G$ contains a triangle or $G \in \{Q_3,V_8\}$.
But each $Q_3$ and $V_8$ contains two disjoint cycles with length less than five, so $G$ contains a triangle.
This proves Statement 2.
Statement 3 is an immediate consequence of Lemma \ref{edge property F}.
If $G \not \in \F_{3,3} \cup \F_{4,4}$ and has girth at least five, then either $G \in \F_{4,2} \cup \F_{4,3}$ or $i \geq 5$ by Lemma \ref{basic F}.
But $G$ is $\W$-free, so $r(G)=0$.
Hence this lemma holds for 2-connected graphs $G$ and graphs with at most two vertices.

Now we assume that $G$ is not 2-connected and $G$ contains at least three vertices.
Since $G$ is subcubic and does not contain two disjoint cycles with length less than five, there exists at most one block containing a cycle of length less than five.
If such a block exists, we denote it by $B$.

Let $D$ be a block of $G$ such that $D$ has girth at least five.
Suppose $\epsilon(D)>0$.
So $D \in \F_{i,j}$ for some $1 \leq i \leq 4,j \leq i$.
Since $G$ is not 2-connected, $D$ contains at least one vertex of degree at most two, so $i-j \geq 1$.
Since $D$ has girth at least five, $i = 4$ and $j \in \{2,3\}$ by Lemma \ref{basic F}, so $\epsilon(D) \leq \frac{1}{7}$ and $D$ isomorphic to a member in $\F_{4,2} \cup \F_{4,3}$.
Since $D$ has girth at least five, $D$ is isomorphic to a member of $\W$.
But $G$ is $\W$-free, a contradiction.
So every block $D$ of $G$ with girth at least five has $\epsilon(D)=0$.
Hence, if $G$ has girth at least five, then $r(G)=0$; otherwise, $B$ exists and $r(G)=\epsilon(B)$.

Since $B$ has at most two vertices or is a 2-connected graph, the lemma holds for $B$, so $r(G) \leq \frac{4}{7}$.
And if $G \neq K_4$, then $B \neq K_4$, so $r(G) \leq \frac{3}{7}$.
Finally, if $r(G)=\frac{3}{7}$, then $r(B)=\frac{3}{7}$, so $B$, and hence $G$, contains a triangle.
\end{proof}

The following will be used in Section \ref{sec:subcubic proof}.

\begin{lemma} \label{-e girth 5}
Let $G$ be a graph isomorphic to a member of $\F_{i,j}$ for some nonnegative integers $i,j$.
If $G$ is planar and there exists an edge $e$ of $G$ such that $G-e$ has girth at least five, then either $i\geq 4$ or $j=0$.
\end{lemma}

\begin{proof}
Since $G \in \F_{i,j}$ and $G$ is simple, $G$ is 2-connected, so $G-e$ is connected.
We may assume that $j>0$.
By Lemma \ref{basic F}, $\lvert E(G-e) \rvert=i+5j-1 \geq i+3j=\lvert V(G) \rvert$, so $G-e$ is not a tree.
Hence every face is incident with all edges of some cycle.
Therefore, every face has length at least five.
By Euler's formula, $|E(G)| -1 \leq \frac{5}{3}(|V(G)| - 2)$.  
By Lemma \ref{basic F}, $i + 5j - 1 \leq \frac{5}{3}(i + 3j - 2)$.
Since $i$ is an integer, $i\geq 4$.
\end{proof}


\section{Proof of Theorem \ref{strong subcubic}}\label{sec:subcubic proof}

In this section, we shall prove Theorem \ref{strong subcubic}.
Let $\S$ be the set of graphs $H$ such that $H \in \F_{3,2} \cup \F_{3,3}$ and there exists an edge $e$ of $H$ such that $H-e$ has girth at least five.
We define $\F$ to the family of graphs consisting of the members of $\F_{4,2} \cup \F_{4,3}$ with girth at least five and the members of $\S$.
Note that every member of $\F$ is non-planar by Lemmas \ref{girth 5 planar in F} and \ref{-e girth 5}.
To prove Theorem \ref{strong subcubic}, by Lemma \ref{r is good}, it is sufficient to show that $\phi(G) \leq \frac{2\lvert E(G) \rvert-\lvert V(G) \rvert+2}{7}+r(G)$ for every connected subcubic graph with no induced subdivision of members of $\F$ and no two disjoint cycles of length less than five.

In the rest of this section, we assume that $G$ is a minimum counterexample.
That is, $G$ is a connected subcubic graph with no induced subdivision of members of $\F$ and no two disjoint cycles of length less than five and $\phi(G) > \frac{2\lvert E(G) \rvert-\lvert V(G) \rvert+2}{7}+r(G)$, but $\phi(H) \leq \frac{2\lvert E(H) \rvert-\lvert V(H) \rvert+2}{7}+r(H)$ for every connected subcubic graph with $\lvert V(H) \rvert + \lvert E(H) \rvert < \lvert V(G) \rvert + \lvert E(G) \rvert$.
We denote $\lvert V(G) \rvert $ by $n$ and denote $\lvert E(G) \rvert$ by $m$.

Clearly, $G$ is not a tree.
So $n \geq 3$.

\begin{claim} \label{subcubic 2-conn}
$G$ is 2-connected.
\end{claim}

\begin{proof}
Suppose that $G$ contains a leaf $v$.
Then $\phi(G) \leq \phi(G-v) \leq \frac{2(m-1)-(n-1)+2}{7}+r(G-v)< \frac{2m-n+2}{7}+r(G)$, a contradiction.
So $G$ has minimum degree at least two.

Suppose that $G$ is not 2-connected.
Let $B$ be an end-block of $G$.
Since $n \geq 3$, $B \neq G$.
Since $G$ is subcubic and has minimum degree at least two, $B$ is a component of the graph obtained from $G$ by deleting an edge.

Let $n_B  = |V(B)|$ and $m_B = |E(B)|$.
By the minimality of $G$, $B$ contains a feedback vertex set $S_B$ of size at most $\frac{2m_B - n_B + 2}{7} + \epsilon(B)$, and $G-V(B)$ contains a feedback vertex set $S_B'$ of size at most 
$\frac{2(m - m_B - 1) - (n - n_B) + 2}{7} + r(G-V(B))$. 
Then $S_B\cup S_B'$ is a feedback vertex set of $G$ of size at most $\frac{2m - n + 2}{7} + r(G-V(B)) + \epsilon(B)= \frac{2m - n + 2}{7} + r(G)$,
a contradiction.
This proves the claim.
\end{proof}

\begin{claim} \label{r=0}
$r(G)=0$.
\end{claim}

\begin{proof}
By Claim \ref{subcubic 2-conn}, $r(G) = 0$ unless $G \in \F_{i,j}$ for some nonnegative integers $i,j$ with $i \leq 4$.
But $G \not \in \F_{i,j}$ for all nonnegative integers $i,j$ with $i \leq 4$ by Lemma \ref{vertex property F}.
\end{proof}

An {\em edge-cut} of $G$ is an ordered partition $[A,B]$ of $V(G)$.
The {\em order} of $[A,B]$ is the number of edges with one end in $A$ and one end in $B$.

\begin{claim}\label{no splitter with B in F}
If $[A,B]$ is an edge-cut of order at most two, then $G[A],G[B]\notin\F_{i,j}$ for all integers $i,j$ with $1 \leq i \leq 4$ and $0 \leq j \leq i$.
\end{claim}

\begin{proof}
Suppose to the contrary.
By symmetry, we may suppose that $G[B] \in \F_{i,j}$ for some $i,j$ with $1 \leq i \leq 4$ and $0 \leq j \leq i$.
Since $G$ has no loops or parallel edges, $\lvert B \rvert \geq 2$.
Since $G$ is 2-connected, $[A,B]$ has order two, so $G[B]$ contains at least two vertices of degree at most two.
So $i \geq 3$ and $i-j \geq 2$.
Let $u_A,v_A$ be the ends of the edges between $A,B$ in $A$, and let $u_B$ and $v_B$ be the ends of the edges between $A,B$ in $B$.
Note that $u_B \neq v_B$ as $G$ is 2-connected.
Let $n_A = \lvert A \rvert$, $n_B=\lvert B \rvert$, $m_A = \lvert E(G[A]) \rvert$ and $m_B = \lvert E(G[B]) \rvert$.
By Lemma \ref{vertex property F}, $G[B]$ has a feedback vertex set $S_B$ with size at most $\frac{2m_B-n_B+2}{7}+r(G[B])=\frac{2m_B-n_B+2+(5-i)}{7}$ such that $u_B \in S_B$.

We first assume that $u_A=v_A$.
Since $G$ is 2-connected, $n_A=1$.
Note that $S_B$ is a feedback vertex set of $G$ as $u_B \in S_B$.
Since $i \geq 3$, $S_B$ is a feedback vertex set of $G$ with size at most $\frac{2m-n+2}{7}$, a contradiction.
So $u_A \neq v_A$.

By the minimality of $G$, $G[A]$ admit s feedback vertex set $S_A$ of size at most $\frac{2m_A - n_A + 2}{7} + r(G[A])$.
Since $u_B \in S_B$, $S_A\cup S_B$ is a feedback vertex set of $G$ of size at most 
$$\frac{2m_A - n_A + 2}{7}+r(G[A]) + \frac{2m_B - n_B + 7-i}{7}= \frac{2(m_A + m_B) - n + 9-i}{7} + r(G[A])$$
$$=\frac{2(m - 2) - n + 9-i}{7} + r(G[A]) = \frac{2m - n+5-i}{7} + r(G[A]),$$
so $r(G[A])+ \frac{3-i}{7}>r(G)=0$. 
Since $i \geq 3$, $r(G[A]) >0$.

Since $i-j \geq 2$ and $G$ is $\F_{4,2}$-free, we have $3 \leq i \leq 4$ and $0 \leq j \leq 1$.
By Lemmas \ref{basic F}, $G[B]$ has girth less than five, so $G[A]$ has girth at least five.

Since $r(G[A])>0$, $G[A]$ contains a block $D \in \F_{s,t}$ for some integers $s,t$ with $s \leq 4$.
Since $G$ is 2-connected, $D$ contains at least two vertices of degree at most two, so $s-t \geq 2$.
Hence $3 \leq s \leq 4$ and $0 \leq t \leq 1$.
By Lemma \ref{basic F}, $G[A]$ contains a cycle of length less than five, a contradiction.
This proves the claim.
\end{proof}

\begin{claim} \label{one side is not cycle}
If $[A,B]$ is an edge-cut of $G$ of order two, then $G[A]$ is not a cycle.
\end{claim}

\begin{proof}
Suppose that $G[A]$ is a cycle.
Let $v$ be a vertex in $A$ adjacent to some vertex in $B$.
Then $G$ has a feedback vertex set containing $\{v\}$ with size $\phi(G[B])+1 \leq \frac{2(m-\lvert A \rvert-2)-(n-\lvert A \rvert)+2}{7}+r(G[B])+1 = \frac{2m-n-\lvert A \rvert+5}{7}+r(G[B])$.
By Claim \ref{no splitter with B in F}, $r(G[B])=0$ and $\lvert A \rvert \geq 5$.
Hence $\phi(G) \leq \frac{2m-n}{7}$, a contradiction.
\end{proof}

We say an edge-cut $[A,B]$ is a {\em splitter} if it has order at most two with $\lvert A \rvert \geq 2$ such that $G[A]$ has girth at least five and $G[B]$ is 2-connected.
Note that $A,B$ have size at least two, so $[A,B]$ has order two and the edges of $[A,B]$ do not share ends, since $G$ is 2-connected.
A splitter is {\em tight} if there exists no splitter $[A',B']$ such that $B' \subset B$.

\begin{claim}\label{splitter 1}
If $[A,B]$ is a tight splitter of $G$, then either
	\begin{enumerate}
		\item $G[A]$ is a path,	
		\item $\phi(G[A]) \leq \frac{2|E(A)| - |V(A)| - 2}{7}$.
	\end{enumerate}
\end{claim}

\begin{proof}
Let $[A, B]$ be a tight splitter.  
Let $u_A,v_A$ be the ends of the edges between $A,B$ in $A$, and let $u_B$ and $v_B$ be the ends of the edges between $A,B$ in $B$.
Let $n_A = \lvert A \rvert$, $n_B=\lvert B \rvert$, $m_A = \lvert E(G[A]) \rvert$ and $m_B = \lvert E(G[B]) \rvert$.
By the minimality of $G$, $G[A]$ and $G[B]$ admit feedback vertex sets $S_A$ and $S_B$ of size at most $\frac{2m_A - n_A + 2}{7} + r(G[A])$ and $\frac{2m_B - n_B + 2}{7} + \epsilon(G[B])$, respectively.  
By Claim \ref{no splitter with B in F}, $\epsilon(G[B]) = 0$.

First assume that $u_A$ is not adjacent to $v_A$.  
Let $A' = G[A]+u_Av_A$.  
Note that $A'$ is 2-connected and has no disjoint cycles of length less than five.  
Since $B$ is connected, $A'$ does not contain any induced subdivision of members of $\F$.
By the minimality of $G$, $A'$ admits a feedback vertex set $S_{A'}$ of size at most $\frac{2(m_A + 1) - n_A + 2}{7} + \epsilon(A') = \frac{2m_A - n_A + 4}{7} + \epsilon(A')$.  Then $S_{A'}\cup S_B$ is a feedback vertex set of $G$ of size at most 
$$\frac{2(m - 2) - n + 6}{7} + \epsilon(A') = \frac{2m - n + 2}{7} + \epsilon(A').$$ 
Therefore $\epsilon(A') > 0$, so $A'\in\F_{i,j}$ for some $i \leq 4$.  

Since $A'-u_Av_A$ equals $G[A]$ and has girth at least five, $i \geq 3$ and $(i,j) \neq (3,1)$ by Lemma \ref{basic F}.
Furthermore, $A' \not \in \F_{3,2} \cup \F_{3,3}$, otherwise $A' \in {\mathcal S} \subseteq \F$ and $G$ contains $A'$ as an induced subdivision. 
Hence, either $i= 4$ or $j=0$.  
If $j=0$, then $G[A]$ is a path, as desired.  
If $i= 4$, then $\epsilon(A') \leq \frac{1}{7}$.  
By Lemma \ref{edge property F}, $\phi(G[A]) \leq \frac{2(m_A+1) - n_A + 2}{7} + \frac{1}{7} - 1 \leq \frac{2m_A - n_A - 2}{7}$, as desired.

Therefore we may assume $u_A$ is adjacent to $v_A$.  
If $u_A$ or $v_A$ has degree two in $G$, then $G[A]$ is isomorphic to $K_2$, as desired.  
So we may assume $u_A$ and $v_A$ have degree three in $G$.  
Let $u'_A$ be the neighbor of $u_A$ in $A$ other than $v_A$, and let $v_A'$ be the neighbor of $v_A$ in $A$ other than $u_A$.  
Suppose $u'_A$ has degree two in $G$.  Let $A'' = G[A-\{u'_A,u_A,v_A\}]$.  By the minimality of $G$, $A''$ admits a feedback vertex set $S_{A''}$ of size at most 
$$\frac{2(m_A - 4) - (n_A - 3) + 2}{7} + r(A'') = \frac{2m_A - n_A - 3}{7} + r(A'').$$  
Since $A''$ has girth at least five and $G$ is 2-connected, if $r(A'')>0$, then $A''$ contains a block with girth at least five belonging to $\F_{i',j'}$ for some $i',j'$ with $i' \leq 4$ and $j'$ with $i'-j' \geq 2$, which implies that $(i',j')=(4,2)$ by Lemma \ref{basic F}, a contradiction.  
So $r(A'')=0$.
Then $S_{A''}\cup S_B\cup\{u_A\}$ is a feedback vertex set of $G$ of size at most
$\frac{2(m - 2) - n - 3 + 2}{7} + 1 = \frac{2m - n + 2}{7}$, a contradiction.  

Therefore we may assume $u'_A$ has degree three in $G$.  
Similarly, we may assume that $v_A'$ has degree three in $G$.
Let $x$ and $y$ be the neighbors of $u'_A$ distinct from $u_A$.  
Note that $x$ and $y$ are non-adjacent vertices distinct from $v_A$ since $G[A]$ has girth at least five.  
Let $A''' = G[A\backslash\{u'_A,u_A,v_A\}]+xy$.  
By the minimality of $G$, $A'''$ admits a feedback vertex set $S_{A'''}$ of size at most
$$\frac{2(m_A - 4) - (n_A - 3) + 2}{7} + r(A''') = \frac{2m_A - n_A - 3}{7} + r(A''').$$
Then $S_{A'''}\cup S_B\cup\{u_A\}$ is a feedback vertex set of $G$ of size at most $\frac{2m - n + 2}{7} + r(A''')$.  
Hence, $r(A''')>0$.  
So $A'''$ contains a block $D$ belonging to $\F_{s,t}$ for some $s \leq 4$.

If $D$ does not contain $xy$, then $D$ has girth at least five and contains at least two vertices of degree at most two, so $D \in \F_{4,2}$ and $D$ is an induced subgraph of $G$, a contradiction.
So $D$ contains $xy$, and it is the unique block of $A'''$ with $\epsilon>0$.
Hence $r(A''') = \epsilon(D)$.

Note that $D-xy$ has girth at least five, so $s \geq 3$ and $(s,t) \neq (3,1)$.
Note that $D \not \in \F_{3,2} \cup \F_{3,3}$, otherwise $G$ contains an induced subdivision of a member of $\F$.
Furthermore, $D \not \in \F_{3,0}$, otherwise $G[V(D) \cup \{u_A'\}]$ is a 4-cycle in $G[A]$, a contradiction.
So $s=4$ and $r(A''')=\frac{1}{7}$.

If $D=A'''$, then $\phi(A'''-xy) \leq\frac{2(m_A-4) - (n_A-3) + 2}{7} + r(A''') - 1$ by Lemma \ref{edge property F}; otherwise, since $G$ is 2-connected and $v_A'$ has degree at least three in $G$, $D$ is an end-block of $A'''$ and $A'''-V(D)$ has no leaves, so $\phi(A'''-xy) = \phi(D-xy)+ \phi(A'''-V(D)) \leq \frac{2\lvert E(D) \rvert - \lvert V(D) \rvert+2}{7}+\epsilon(D)-1+\frac{2(\lvert E(A''') \rvert - \lvert E(D) \rvert-1)+(\lvert V(A''') \rvert-\lvert V(D) \rvert)+2}{7}+r(A'''-D) = \frac{2(m_A-4)-2-(n_A-3)+4}{7}+\frac{1}{7}-1$.
In either case, $\phi(A'''-xy) \leq \frac{2(m_A-4) - (n_A-3) + 2}{7} + \frac{1}{7} - 1 = \frac{2m_A-n_A-9}{7}$.

Note that $A'''-xy=G[A-\{u_A,v_A,u_A'\}]$, so by adding $u_A'$ into a minimum feedback vertex set of $A'''-xy$, $G[A]$ has a feedback vertex set with size $\phi(A'''-xy)+1 \leq \frac{2m_A-n_A-2}{7}$, as desired.
This proves the claim.
\end{proof}

\begin{claim} \label{splitter 2}
If $[A,B]$ is a tight splitter of $G$, then $G[A]$ is a path.
\end{claim}

\begin{proof}
Let $u_A,v_A$ be the ends of the edges between $A,B$ in $A$, and let $u_B$ and $v_B$ be the ends of the edges between $A,B$ in $B$.
Let $n_A = \lvert A \rvert$, $n_B=\lvert B \rvert$, $m_A = \lvert E(G[A]) \rvert$ and $m_B = \lvert E(G[B]) \rvert$.

Suppose that $G[A]$ is not a path.
By Claim \ref{splitter 1}, $G[A]$ admits a feedback vertex set $S_A$ of size at most $\frac{2m_A - n_A - 2}{7}$.
Let $H = G[B]-u_B$. 
Since $[A,B]$ is a splitter, $G[B]$ is 2-connected, so $H$ is connected and $u_B$ has degree two in $G[B]$.
By the minimality of $G$, $H$ admits a feedback vertex set $S_H$ of size at most$\frac{2(m_B - 2) - (n_B - 1) + 2}{7} + r(H) = \frac{2m_B - n_B - 1}{7} + r(H)$.
Then $S_A\cup S_H\cup\{u_B\}$ is a feedback vertex set of $G$ of size at most 
$\frac{2(m - 2) - n - 2 - 1}{7} + r(H)+1 = \frac{2m - n}{7} + r(H)$.
Hence, $r(H) > \frac{2}{7}$.  

If $D$ is a block of $H$ with girth at least five and $\epsilon(D)>0$, then $D$ contains at least two vertices of degree two, so $D \in \F_{4,2}$ and $G$ contains an induced subdivision of a member of $\F$, a contradiction.
Since $G$ does not contain two disjoint cycles with length less than five, there exists a unique block $D$ with $\epsilon(D)>0$.
Hence $\epsilon(D) = r(H) \geq \frac{3}{7}$ and $D \in \F_{i,j}$ for some $i,j$ with $1 \leq i \leq 2$ and $0 \leq j \leq i$.
Since $D$ is simple, $D \in \F_{2,1}$.
But every block of $H$ contains at least two vertices of degree at most two, a contradiction.
\end{proof}

\begin{claim} \label{claim:contracting}
If $u$ is a vertex of degree two in $G$, then $G$ contains a cycle $C$ with length five and a cycle $C'$ with length less than five such that $u \in V(C)$ and $V(C) \cap V(C') = \emptyset$.
\end{claim}

\begin{proof}
Let $G'$ be the graph obtained from $G$ by contracting an edge incident with $u$.  
By Claim \ref{no splitter with B in F}, the neighbors of $u$ in $G$ are non-adjacent.
So $G'$ is simple, 2-connected and does not contain an induced subdivision of a member of $\F$.
Suppose that $G'$ does not contain two disjoint cycles of length less than five.
By the minimality of $G$, $G'$ admits a feedback vertex set $S$ of size at most
$$\frac{2(m-1) - (n-1) + 2}{7} + \epsilon(G') = \frac{2m - n + 1}{7} + \epsilon(G').$$
Note that $S$ is a feedback vertex set of $G$.  
Hence $\epsilon(G') > \frac{1}{7}$, so $G'\in\F_{i,j}$ for some $i,j$ with $i\leq 3$.  
Since $G$ can be obtained from $G'$ by subdividing one edge, $G\in\F_{i+1,j}$, so $r(G) \geq \frac{1}{7}$, contradicting Claim \ref{r=0}.
Therefore, $G'$ contains two disjoint cycles of length less than five, where one of them contains $u$.
So $G$ contains two disjoint cycles, where one of them contains $u$ and has length five, and the other has length less than five.
\end{proof}

We say that $G$ is {\em internally $3$-edge-connected} if $G$ is 2-edge-connected and for every edge-cut $[A,B]$ of $G$ of order two, $\lvert A \rvert=1$ or $\lvert B \rvert=1$.

\begin{claim} \label{claim:3-connected}
$G$ is internally 3-edge-connected.
\end{claim}

\begin{proof}
Suppose that $G$ is not internally 3-edge-connected.
We define $[A,B]$ to be an edge-cut of $G$ of order two such that $\lvert A \rvert \geq 2, \lvert B \rvert \geq 2$ and $G[A]$ has girth at least five, and subject to that, $B$ is minimal.
Note that such an edge-cut $[A,B]$ exists since $G$ is not internally 3-edge-connected, so there exists an edge-cut of order two such that both sides contain at least two vertices and one side has girth at least five since $G$ does not contain two disjoint cycles of length less than five.
Let $u_A,v_A$ be the ends of the edges between $A,B$ in $A$, and let $u_B$ and $v_B$ be the ends of the edges between $A,B$ in $B$.
Notice that $u_A \neq v_A$ and $u_B \neq v_B$ since $\lvert A \rvert \geq 2$ and $\lvert B \rvert \geq 2$.
Let $n_A = \lvert A \rvert$, $n_B=\lvert B \rvert$, $m_A = \lvert E(G[A]) \rvert$ and $m_B = \lvert E(G[B]) \rvert$.

We claim that $G[B]$ is 2-connected.
Suppose that $G[B]$ is not 2-connected.
Since $G[A]$ has girth at least five, $\lvert B \rvert \geq 3$ by Claim \ref{claim:contracting}.
If $G[B]$ contains a leaf $v$, then $v \in \{u_B,v_B\}$ since $G$ is 2-connected.
Then $[A \cup \{v\}, B-\{v\}]$ is an edge-cut of order two with $G[A \cup \{v\}]$ of girth at least five, contradicting the minimality of $B$.
So $G[B]$ has no leaf, and hence $G[B]$ contains at least two end-blocks, where each of them has minimum degree two.
Since $G$ has no two disjoint cycles with length less than five, there exists an end-block $B'$ of $G[B]$ such that $B'$ has girth at least five.
Since $G$ is 2-connected, $B'$ contains $u_B$ or $v_B$.  Therefore $[A \cup V(B'), B-V(B')]$ is an edge-cut of $G$ of order two such that $G[A \cup V(B')]$ has girth at least five, contradicting the minimality of $B$.
This proves that $G[B]$ is 2-connected.

Therefore, $[A,B]$ is a splitter. 
The minimality of $B$ implies that $[A,B]$ is tight and $u_B$ and $v_B$ have degree three.
By Claim \ref{splitter 2}, $G[A]$ is a path.  

Suppose that $u_B$ is adjacent to $v_B$.
Then $G[A\cup\{u_B,v_B\}]$ is a cycle.
Since $G$ is not a cycle, $[V(G-A-\{u_B,v_B\}), A\cup\{u_B,v_B\}]$ is an edge-cut of order two, where $G[A\cup\{u_B,v_B\}]$ is a cycle, contradicting Claim \ref{one side is not cycle}.

So $u_B$ is not adjacent to $v_B$.
Therefore by Claim \ref{claim:contracting}, $G[A]$ is isomorphic to $K_2$, $u_B$ and $v_B$ share a neighbor in $B$ which we call $w$, and there exists a cycle $C$ in $G[B]$ of length less than five disjoint from $\{u_B,v_B,w\}$.
By the minimality of $B$, $w$ has degree three in $G$. 

Let $H = G[B]-u_B$.  Then $H$ admits a feedback vertex set $S_H$ of size at most $\frac{2(m - 5) - (n - 3) + 2}{7} + r(H) = \frac{2m - n - 5}{7} + r(H)$.
Then $S_H\cup\{u_B\}$ is a feedback vertex set of $G$ of size at most $\frac{2m - n + 2}{7} + r(H)$.  
Hence, $r(H) > 0$.  
Since $G[B]$ is 2-connected and $u_B$ has degree two in $G[B]$, each neighbor of $u_B$ is contained in an end-block of $H$ and is not a cut-vertex of $H$.
So $w$ is contained in an end-block of $H$, and no other block of $H$ contains $w$.
But $w$ has degree two in $H$, so both neighbors of $w$ in $H$ are contained in the same block of $H$ as $w$.
Since $v_B$ is adjacent to $w$, $w$ and $v_B$ are in the same end-block of $H$. 
Let $D$ be the end-block of $H$ containing $w$ and $v_B$.
Since $[A,B]$ is a tight splitter, $D$ contains a cycle of length less than five, and hence $C$ is contained in $D$.

If $H$ is 2-connected, then $\epsilon(D)=r(H)>0$.
If $H$ is not 2-connected, then since there are at most three edges in $G$ between $V(H)$ and $V(G)-V(H)$, for every block $D'$ of $H$ other than $D$, $[V(G)-V(D'),V(D')]$ is an edge-cut of order two, so $\epsilon(D')=0$ by Claim \ref{no splitter with B in F}.
Hence $\epsilon(D)=r(H)>0$ in either case.
Therefore, $D \in \F_{i,j}$ for some $i,j$ with $1 \leq i \leq 4$ and $0 \leq j \leq i$.  
Note that $D$ contains three vertices of degree two in $D$, even if $H$ is not 2-connected.
So $i - j\geq 3$, and hence $i\geq 3$ and $j\leq 1$.  
If $j=0$, $D$ is a triangle or 4-cycle, so $D=C$.
But $C$ is disjoint from $w$, a contradiction.
So $i=4$ and $j=1$.
Therefore, $D$ is a subdivision of $K_4$ with three vertices of degree two.
We denote the vertex of degree two in $D$ other than $w$ and $v_B$ by $z$.

Suppose that $D$ is the unique non-trivial block of $H$.
Since $H$ has at most two end-blocks, either $H=D$, or $H$ is obtained from $D$ by attaching a path $P$.
If $H=D$, then let $H'=G$; otherwise, let $H'$ be the multigraph obtained from $G$ by suppressing all vertices in $P$.
Note that in the latter case, $H'$ is the multigraph $G[A \cup V(D)\cup\{u_B\}]+u_Bz$.
So $H'$ is the same graph in either case.
But $H'$ can be obtained from a subdivision of $K_4$ by applying operation $\circ$ and then subdividing edges, so $H' \in \F_{i',j'}$ for some nonnegative integers $i',j'$ with $j' \leq i'$.
However, $G$ is a subdivision of $H'$, so $G \in \F_{i'',j''}$ for some nonnegative integers $i'',j''$, a contradiction.
Therefore, $H$ contains a non-trivial block $W$ other than $D$.

However, $[V(W),V(G)-V(W)]$ is an edge-cut of $G$ of order two such that both sides contain at least two vertices and $W$ has girth at least five.
So there exists an edge-cut $[A'',B'']$ of $G$ of order two such that both sides contains at least two vertices, $V(W) \subseteq A''$, $G[A'']$ has girth at least five, and subject to that, $B''$ is minimal.
Then $[A'',B'']$ is a tight splitter.
Claim \ref{splitter 2} implies that $G[A'']$ is a path.
But $W$ is a subgraph of $G[A'']$, a contradiction.
This completes the proof.
\end{proof}

The following two claims were proved in \cite{kl}.

\begin{claim}[\cite{kl}] \label{claim:menger}
If $u,v$ are vertices of $G$ of degree three, then there exist three internally disjoint paths from $u$ to $v$.
\end{claim}

\begin{claim}[\cite{kl}] \label{claim:remove}
If $X\subseteq V(G)$ and $\lvert E(X, V(G-X)) \rvert = 3$, then $G-X$ is connected and has at most one nontrivial block.  
Furthermore, if $u,v\in V(G-X)$ are adjacent to vertices in $X$ but $uv \not \in E(G)$, then $(G- X) + uv$ has at most one nontrivial block.
\end{claim}

\begin{claim} \label{claim:notriangle}
$G$ is triangle-free.
\end{claim}

\begin{proof}
Suppose $a,b,$ and $c$ are the vertices of a triangle in $G$.  Since $G$ is not a triangle, all have degree three by Claim \ref{claim:3-connected}.
Let $X = \{a,b,c\}$.  
Let $a',b',$ and $c'$ be the neighbors of $a,b,$ and $c$ not in $X$ respectively.
If $a',b',$ and $c'$ are not pairwise distinct, then by Claim \ref{claim:3-connected}, $G \in \{K_4,K_4^+\} \subseteq \F_{1,1} \cup \F_{2,1}$, a contradiction.
Therefore $a',b'$, and $c'$ are pairwise distinct.
If $a',b',$ and $c'$ are pairwise adjacent, then $G$ contains two disjoint triangles, a contradiction.
Suppose one of $a',b',c'$ is adjacent to the other two.  We may assume without loss of generality that $a'$ is adjacent to $b'$ and $c'$.  
By Claim \ref{claim:3-connected}, either $V(G) = X \cup \{a',b',c'\}$ or $G-(X \cup \{a',b',c'\})$ is an isolated vertex.  
In the first case, $G\in \F_{3,1}$, a contradiction.   In the second case, $G$ contains a triangle disjoint from a 4-cycle, a contradiction.

Therefore, $|E(G[\{a',b',c'\}]) | \leq 1$.
By symmetry, we may assume that $b'$ is not adjacent to $a'$ or $c'$.
Let $H=(G- X)+b'c'$.
By Claim \ref{claim:remove}, $H$ is connected and has at most one nontrivial block, denoted by $B$, so $r(H)=\epsilon(B)$.  
Since $G- X$ has girth at least five, $H$ contains no disjoint cycles of length less than five.
Note that $H$ does not contain an induced subdivision of a member of $\F$, so $H$ admits a feedback vertex set $S$ of size at most 
$$\frac{2(m - 5) - (n - 3) + 2}{7} + \epsilon(B) = \frac{2m - n + 2}{7} - 1 + \epsilon(B)$$
Then $S\cup\{a\}$ is a feedback vertex set of $G$ of size at most $\frac{2m - n + 2}{7} + \epsilon(B)$ vertices, so $\epsilon(B)> 0$.
Therefore $B\in \F_{i,j}$ for some $i,j$ with $1 \leq i \leq 4$ and $0 \leq j \leq i$.

Note that $G=H \circ (b'c',b'c',a')$.  So if $H$ is 2-connected, then $G\in\F_{i,j+1}$, a contradiction.  
Therefore $H$ is not 2-connected.  
Hence, $a'$ is a leaf of $H$.
Since $a',b',c'$ are the only vertices adjacent to some vertex in $X$ and $b'c' \in E(H)$, $B=H-a'$.
Let $a''$ be the neighbor of $a'$ other than $a$.
Then $G$ can be obtained from $B \circ (b'c',b'c',a'')$ by subdividing an edge.
Hence, $G\in\F_{i+1,j+1}$, a contradiction.
\end{proof}

\begin{claim} \label{cubic in short cycle}
No vertex in a cycle of $G$ of length less than five has degree two in $G$.
\end{claim}

\begin{proof}
Let $D$ be a cycle of length less than five containing a vertex $v$ of degree two in $G$.
By Claim \ref{claim:notriangle}, $D$ is a 4-cycle.
By Claim \ref{claim:contracting}, there exists a cycle $C$ of length five containing $v$ and a cycle $C'$ with length less than five disjoint from $C$.
Since $v$ has degree two, $D$ shares at least three vertices with $C$.
Since $G$ is subcubic, $C'$ and $D$ are disjoint cycles of length less than five, a contradiction.
\end{proof}

\begin{claim}\label{claim:op2makesbad}
Let $a,b$, and $c$ be distinct vertices of degree three in $G$ such that $abc$ is a path in $G$.  
Let $a_1$ and $a_2$ be the neighbors of $a$ other than $b$, let $c_1$ and $c_2$ be the neighbors of $c$ other than $b$, and let $b'$ be the other neighbor of $b$.  
If $G'= (G-\{a,b,c\})+a_1a_2+c_1c_2$, then $G'$ contains two disjoint cycles of length less than five.
\end{claim}

\begin{proof}
Note that $a_1a_2,c_1c_2 \not \in E(G)$ and $\{a_1,a_2,c_1,c_2\} \cap \{a,b,c\}=\emptyset$ since $G$ is triangle-free.
Then $G=G' \circ (a_1a_2,c_1c_2,b)$, so $G'$ is connected and $G' \notin\F_{i,j}$ for any integers $i,j$.
Note that $G'$ does not contain an induced subdivision of a member of $\F$ as $G$ does not.
By Claim \ref{claim:remove}, $G-b$ has at most one nontrivial block.  
Since $G$ is triangle-free, $G'$ has at most one nontrivial block, denoted by $B$.
Since $G$ has minimum degree at least two, $b'$ is the only possible vertex in $G'$ that has degree less than two.
Since $G$ is internally 3-edge-connected, $G$ has no two adjacent vertices of degree at most two.
So either $G'=B$ or $G'$ is obtained from $B$ by attaching a leaf $b'$.

Suppose that $G'$ does not contain two disjoint cycles of length less than five.
Then by the minimality of $G$, $G'$ admits a feedback vertex set $S$ of size at most $\frac{2(m - 5) - (n - 3) + 2}{7} + r(G') = \frac{2m - n + 2}{7} - 1 + r(G')$.
But $S\cup\{b\}$ is a feedback vertex set of $G$ of size at most $\frac{2m - n + 2}{7} + r(G')$, so $r(G') > 0$.  
If $G'=B$, then $B\in\F_{s,t}$ for some integers $s,t$ with $s \geq 1$ and $0 \leq t \leq s$, a contradiction.  
So $G'$ is obtained from $B$ by attaching the leaf $b'$.  
Let $b''$ be the neighbor of $b'$ in $B$.
But then $G$ is obtained from $B\circ(a_1a_2,c_1c_2,b'')$ by subdividing an edge, so $G\in\F_{s+1,t+1}$, a contradiction.
\end{proof}

\begin{claim} \label{girth 5}
$G$ has girth at least five.
\end{claim}

\begin{proof}
Suppose to the contrary that $G$ contains a cycle $C$ of length less than five. 
Since $G$ is triangle-free, $C$ is a 4-cycle.  
Let $C=abcda$.
By Claim \ref{cubic in short cycle}, $a,b,c$ and $d$ have neighbors $a',b',c',$ and $d'$ not in $C$, respectively.
Since $G$ is triangle-free, $\{a',c'\}$ is disjoint from $\{b',d'\}$.

Suppose that $a'=c'$ and $b'=d'$.
If $a'$ is adjacent to $b'$, then $G=K_{3,3}$ and $\phi(G)=2=\frac{2m-n+2}{7}$, a contradiction.
So $a'$ is not adjacent to $b'$.
Since $G$ is internally 3-edge-connected, there exists a vertex $v$ of degree two in $G$ adjacent to $a'$ and $b'$, so $G$ can be obtained from $K_{3,3}$ by subdividing an edge. 
Hence $\phi(G)=2<\frac{2m-n+2}{7}$, a contradiction.

Therefore we may assume without loss of generality that $a' \neq c'$.  
Since $G$ is triangle-free, $d$ is not adjacent to $a'$ or $c'$.
Let $G' = (G-\{a,b,c\})+\{a'd, c'd\}$.
By Claim \ref{claim:op2makesbad}, $G'$ contains two disjoint cycles $D_1,D_2$ of length less than five.
Since $G$ does not contain two such cycles, one of them, say $D_1$, contains at least one edge in $\{a'd,c'd\}$.
Therefore $d\in V(D_1)$.  
Since $D_2$ is in $G'$ and is disjoint from $D_1$, $D_2$ does not contain any of $\{a,b,c,d\}$.
Therefore $C$ and $D_2$ are two disjoint cycle of length less than five in $G$, a contradiction.
\end{proof}

Note that Claims \ref{claim:contracting} and \ref{girth 5} imply that $G$ is cubic.

\begin{claim} \label{claim:vertexadjacentto5-cycles}
Let $v\in V(G)$, and let $a$ and $b$ be two distinct neighbors of $v$.  
Let $a_1$ and $a_2$ be the neighbors of $a$ other than $v$, and let $b_1$ and $b_2$ be the neighbors of $b$ other than $v$.  
Then $G-v$ contains two disjoint 5-cycles, where one contains the path $a_1aa_2$ and the other contains the path $b_1bb_2$.
\end{claim}

\begin{proof}
Let $G' = (G-\{a,v,b\})+a_1a_2+b_1b_2$.
By Claim \ref{claim:op2makesbad}, there are two disjoint cycles in $G$ of length less than five, where one of them contains $a_1a_2$ and the other contains $b_1b_2$, since $G$ has girth at least five.
So $G$ contains two disjoint cycles of length at most five, where one contains $a_1aa_2$ and the other contains $b_1bb_2$.
These two cycles have length five since $G$ has girth at least five.
\end{proof}

By \cite[Theorem 6.1]{kl}, $G$ is the dodecahedron, which is in $\F_{5,5}$, a contradiction.  
This completes the proof of Theorem \ref{strong subcubic}.

\bigskip

\noindent{\bf Acknowledgement.}
The authors thank the anonymous referees for careful reading and suggestions.


\end{document}